\documentclass[a4paper]{amsproc}
\usepackage{amssymb}
\usepackage[hyphens]{url} 
\usepackage{amsmath}
\usepackage{cite}
\usepackage{mathrsfs}
\usepackage[title]{appendix} 
\usepackage[utf8]{inputenc}
\usepackage[english]{babel}
\usepackage{hyperref}
\usepackage{amsthm}
\usepackage{tikz-cd}
\usepackage{caption}
\usepackage{float}

\usepackage{extarrows}
\usepackage{graphicx}
\usepackage{enumitem}
\usepackage{comment}
\usepackage{wasysym}
\hypersetup{linktoc=page}

\theoremstyle{plain}
\newtheorem{Theorem}{Theorem}

\newtheorem{Proposition}[Theorem]{Proposition}

\newtheorem{Corollary}[Theorem]{Corollary}

\theoremstyle{remark}
\newtheorem*{Remark}{Remark}

 \numberwithin{Theorem}{section}

\title{Self-similarity of $p$-adic groups}

\author[Amir Y. Weiss Behar and Devora Zalaznik]{\bfseries Amir Y. Weiss Behar and Devora Zalaznik}

\address{
Einstein Institute of Mathematics \\
Edmond J. Safra Campus (Givat-Ram)\\
The Hebrew University of Jerusalem\\
9190401\\
Israel}

\email{amir.behar@mail.huji.ac.il}
\email{devora.zalaznik@mail.huji.ac.il}

\vspace{18mm}  

\begin{document}

\today

\begin{abstract}
We show that a compact open subgroup $H$ of a simple algebraic $p$-adic group $G$ is self-similar if and only if it is isotropic. 
\end{abstract}

\maketitle

\section{Introduction}

Let $X$ be a finite alphabet and $X^{\ast}$ the set of all finite words on $X$. The set $X^{\ast}$ has a natural structure of a rooted tree once declaring that $v,w \in X^{*}$ are adjacent if either $v=wx$ or $w=vx$ for some $x\in X$. \newline
A group G is called  \textit{self-similar} if it acts faithfully on such a tree $X^{\ast}$ satisfying: $(i)$ the action is transitive on $X$; and $(ii)$ for every $g\in G$ and every $x\in X$ there exists $h\in G$ and $y\in X$ such that $g(xw)=yh(w)$ for every word $w\in X^{*}$ (see \S\ref{pre}). \newline

Self-similar groups form a rich and interesting class of groups that got a considerable amount of attention (cf. \cite{Gr80}, \cite{GS83}, \cite{Ne05}  and the references therein). \newline

In \cite{NS22} and \cite{NS23} Noseda and Snopce initiated the study of self similarity of compact $p$-adic analytic groups. In particular, in \cite{NS23} they proved that if $D$ is a finite dimensional noncommutative central division $\mathbb{Q}_{p}$-algebra, and $H$ an open subgroup of $SL(1,D)$, then $H$ is not self-similar. They conjectured that the same assertion holds even if $D$ is central over some $p$-adic field $k$, $[k:\mathbb{Q}_{p}]<\infty$ (conjecture D there). We prove that conjecture, moreover we prove also the converse.

\begin{Theorem}[Main theorem]\label{Mai~The}
    Let $k$ be a $p$-adic field (i.e. a finite extension of $\mathbb{Q}_{p}$), $\mathbf{G}$ a simply connected, absolutely almost simple linear algebraic $k$ group and $H \subseteq \mathbf{G}(k)$ a compact open subgroup. Then $H$ is  self-similar if and only if $\mathbf{G}$ is $k$-isotropic (i.e. $\text{rank}_{k}\mathbf{G} \geq 1$).
\end{Theorem}

Note that the non-isotropic simple algebraic groups $\mathbf{G}$ over a $p$-adic field $k$ (i.e., those with $\text{rank}_k\mathbf{G} = 0$) are exactly $SL(1,D)$ as above. So our theorem gives the complete answer for them, as well as proving the converse.\newline

The paper is organized as follows: In \S2 we give some preliminaries on self-similar actions on rooted trees. In \S3 we prove the extension of Noseda and Snopce result to all anisotropic groups. Our proof will be more conceptual and much shorter then theirs, appealing to some standard results in the theory of division algebras and $p$-adic algebraic groups. This last theory will serve us in \S4 to prove the converse. It will be clear from the proof of both parts that the main difference between the anisotropic case and the isotropic case is the existence (in the isotropic case) of unbounded inner automorphisms.\newline

\textbf{Acknowledgments.} This work is a part of the first author's PhD thesis and the second author's MSc thesis at the Hebrew University.  For suggesting the above topic and for providing helpful guidance, suggestions and ideas both authors are deeply grateful to Alexander Lubotzky and Shahar Mozes. During the period of work on this paper both authors were supported by the European Research Council (ERC) under the European Union’s Horizon 2020 research and innovation programme (grant agreement No 882751), the first author was also supported by the ISF-Moked grant 2019/19.

\section{Preliminaries}
\label{pre}

\subsection{Group actions on rooted trees}

We follow the basic definitions and propositions about self-similar group actions on rooted tree as presented in  \cite{Ne05}. \newline 

Let $X$ be a finite alphabet, and $X^{\ast}$ the rooted tree defined by this alphabet. Its vertices are the finite words on the alphabet $X$ with the special root vertex being the empty word $\emptyset$. Two vertices $v$ and $w$ are connected by an edge if $w=vx$ for some letter $x\in X$. \newline

Let $g : X^{\ast} \to X^{\ast}$  be an endomorphism of the rooted tree  $X^{*}$. For every vertex $v\in X^{\ast}$, one has the associated rooted subtrees $vX^{\ast}$ and $g(v)X^{\ast}$ which are both naturally isomorphic to $X^{\ast}$. Identifying these subtrees with the tree $X^{\ast}$, the restriction $g|_{vX^{\ast}} : vX^{\ast}\to g(v)X^{\ast}$ defines a map $g|_{v}:X^{\ast}\to X^{\ast}$, which is called the restriction of $g$ to $v$ (see Figure \ref{fig~1}). It is uniquely determined
by the condition $ g(vw)=g(v)g|_v(w), \quad \forall w\in X^{\ast} $.

\begin{center}
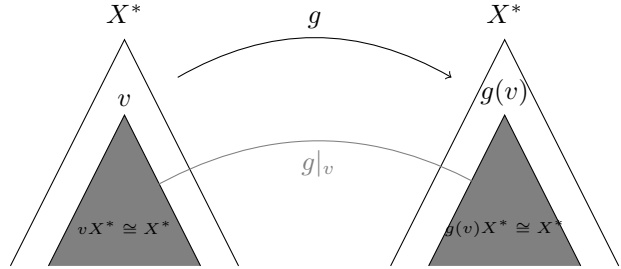
\begin{figure}[b]
\begin{tikzpicture}
    \draw (0,0) -- (1.5,3) -- (3,0);
    \filldraw[fill=gray,opacity=0.3] (0.5,0) -- (1.5,2) -- (2.5,0);
    \node[above] at (1.5,2) {$v$};
    \node[font=\tiny] at (1.5,0.5) {$vX^{\ast}\cong X^{*}$};
    \node[above] at (1.5,3.1) {$X^{\ast}$};
    \draw (5,0) -- (6.5,3) -- (8,0);
    \filldraw[fill=gray,opacity=0.3] (5.5,0) -- (6.5,2) -- (7.5,0);
    \node[above] at (6.5,2) {$g(v)$};
    \node[font=\tiny] at (6.5,0.5) {$g(v)X^{\ast}\cong X^{*}$};
    \node[above] at (6.5,3.1) {$X^{\ast}$};
    \draw[-to] (2.2,2.5) to[bend left] node[above] {$g$} (5.8,2.5);
    \draw[-to,color=gray] (1.8,1.0) to[bend left] node[below] {$g|_{v}$} (6.3,1.0);
\end{tikzpicture}
\caption{The restriction $g|_{v}$}\label{fig~1}
\end{figure}
\end{center}

These restrictions obviously satisfy
\begin{gather*}
    g|_{v_1v_2} = g|_{v_1}|_{v_2}; \\
    (g_1  \cdot g_2)|_v = g_{1}|_{g_2(v)} \cdot  g_{2}|_{v}.
\end{gather*}

\subsection{Self-similar actions and virtual endomorphisms}
A faithful action of a group G on $X^{\ast}$ is said to be \textit{self-similar} if for every $g \in G$ and every $x \in X$ there exist $h \in G$ and $y\in X$ such that $g(xw) = yh(w)$
for every $w \in X^{\ast}$. We denote self similar actions as pairs $(G,X)$ where $G$ is the group and $X$ is the set of alphabet such that G acts on $X^{\ast}$. Since the action is faithful, the pair $(h,y)$ is uniquely determined by the pair $(g,x)$, $y=g(x)$, $h=g|_x$. A group is called \textit{self-similar of index $d$} if it has a faithful self-similar action on a $d$-regular rooted tree which is transitive on the first level of the tree. \newline  
    
A virtual endomorphism $\varphi : G\dashrightarrow G$ is a homomorphism
 $\varphi : G_{0} \rightarrow G$, where $G_{0} \leq G$ is a subgroup of finite index.
The index of the virtual endomorphism is $[G : G_{0}]$. A subgroup $H \leq G$ is said to be $\varphi$-\textit{invariant} if $H \leq G_{0}$ and $\varphi(H)\subseteq H$. A virtual endomorphism is called \textit{simple} if there are no non-trivial normal $\varphi$-invariant subgroup. \newline

The next proposition is a reorganization of some ideas and propositions from \cite[chapter 2]{Ne05}.
\begin{Proposition}
    Let $G$ be a group and $d\geq1$ be an integer. Then $G$ is self-similar of index $d$ if and only if $G$ admits a simple virtual endomorphism of index $d$.
\end{Proposition}

\begin{proof} 
    Assume that $G$ is self-similar of index $d$, i.e. there is a finite alphabet $X$ of cardinality $d$, and a faithful, first level transitive, self-similar action of $G$ on the rooted tree $X^{\ast}$. Let $G_0$ be the stabilizer of $x_{0} \in X$ in $G$, and
    define $\varphi : G_0 \rightarrow G$ by $\varphi(g) := g|_{x_{0}}$. Then for $g_{1},g_{2}\in G_{0}$, $\varphi(g_{1}g_{2})=(g_{1}g_{2})|_{x_0}=g_{1}|_{g_{2}(x_{0})}g_{2}|_{x_{0}}=g_{1}|_{x_{0}}g_{2}|_{x_{0}}=\varphi(g_{1})\varphi(g_{2})$ so $\varphi$ is indeed a virtual endomorphism of $G$ of index  $[G:G_0]$. 
    
    Write $G=\bigsqcup h_{i}G_{0}$. If $h_{i}(x_{0})=h_{j}(x_{0})\in X$ for some $i,j$ then $(h^{-1}_jh_i)(x_{0})=x_{0}$ so $h^{-1}_jh_i\in G_0$. Then $h_i,h_j$ belong to the same left coset of $G_0$, which means there are at most $d$ cosets. On the other hand, the action is transitive, so every $x\in X$ can be written as $(h_i)(x_{0})=h_i(g_0(x_{0}))=(h_ig_0)(x_{0})$ for some $i$ and every $g_0\in G_0$. Thus there are at least d cosets and $[G:G_0]=d.$

    Now let's show that $\varphi$ is simple.
    The action on $X^{\ast}$ is transitive on the first level, so for every $x\in X$ we have $x_{0}=gx$ for some $g\in G$. Then $g^{-1}G_0g=G_x$ which means for every $N \triangleleft G$,  $N\subseteq G_0$ also $N\subseteq G_x.$ So $N$ acts trivially on every $x\in X$. Thus, if  $\varphi(N)\subseteq N$, N acts trivially on every $w\in X^{\ast}$, but the action is faithful so $N={1}.$ \newline

    On the other hand, let $\varphi :G_{0}\rightarrow G$ be a simple virtual endomorphism of index $d$ and let $X=\{0,1,...,d-1\}$. Choose representatives $h_0,h_1,...,h_{d-1}$ such that $G=\sqcup h_{i}G_{0}$. For every $g\in G$ and $i\in X$ there is a unique $j \in X$ such that $h^{-1}_{j}gh_{i} \in G_{0}$. Define $g(i)=j,g|_{i}=\varphi(h^{-1}_{j}gh_{i})$. For every $w\in X^{\ast}$, $g(iw)=g(i)g|_i(w)=j\varphi(h^{-1}_jgh_i)(w)$. This defines an action of $G$ on $X^{\ast}$. We have to show that this action is faithful.
    
    Let $N$ be the subgroup of all $g \in G$ such that $g(w)=w$ for every $w \in X^*$. Indeed, for $g_1,g_2\in N, g_1(g_2(w))=g_2(g_1(w))=w$ and $w=(gg^{-1})(w)=g(g^{-1}(w))=g^{-1}(w)$ so N is a subgroup. Also for every $g\in N, f\in G$ we get that $(f^{-1}gf)(w)=f^{-1}(g(f(w))=w$ so $N$ is normal in $G$. Assume now by contradiction that $N\neq {1}$. Let $i\in X$, by simplicity of $\varphi$ there exists $1\neq g\in N$ such that $\varphi(h_{i}^{-1}gh_{i})\notin N$, and thus we can find some $w\in X^{\ast}$ with $\varphi(h_{i}^{-1}gh_{i})(w)\neq w$.
    Then $g(iw)=g(i)\varphi(h^{-1}_{i}gh_{i})(w)=i\varphi(h^{-1}_{i}gh_{i})(w) \neq iw.$ We get a contradiction so $N={1}$ and the action is faithful. 
    
    It remains to show that the action is transitive on the first level. By definition, if $g(i)=j$ then $h^{-1}_jgh_i \in G_{0}$. Thus, $h^{-1}_{k}h_{k}h^{-1}_{j}gh_{i} \in G_{0}$ for every $k \in X$ and so $(h_{k}h^{-1}_{j}g)(i)=k$. The action is transitive on the first level as needed.
\end{proof}

We call a virtual endomorphism \textit{almost simple} if the only normal invariant subgroups are central. 

\begin{Corollary}
    Let $G$ be a residually finite group with finite center. Then $G$ is self-similar if and only if $G$ admits an almost simple virtual endomorphism.
\end{Corollary}

\begin{proof}
    Assume that $G$ admits an almost simple virtual endomorphism $\varphi: G_{0} \to  G$. Then, as $G$ is residually finite with finite center, there exists a finite index subgroup $G_{1}\subseteq G_{0}$ which intersect the center trivially. Thus $\varphi|_{G_{1}}:G_{1}\to G$ defines a simple virtual endomorphism, as needed.
\end{proof}

\begin{Remark}
    As all groups we consider are residually finite with finite center, in order to prove the main theorem \ref{Mai~The}, it is enough to determine weather they admit an almost simple virtual endomorphism.
\end{Remark}

\section{anisotropic case}

In this section we prove one direction of the main theorem \ref{Mai~The}, i.e. let $k$ be a $p$-adic field and $\mathbf{G}$ a simply connected, absolutely almost simple anisotropic linear algebraic $k$ group and $H \subseteq \mathbf{G}(k)$ a compact open subgroup. Then $H$ has no almost simple virtual endomorphism. \newline

By Tits classification \cite[\S 3.3.3]{Ti66}, there exists a finite dimensional central division algebra $D$ over $k$ such that $\mathbf{G}(k)\cong SL(1,D)$, the group of norm 1 elements of $D$. This norm is defined as follows: Let $K$ be a splitting field for $D$, i.e. there exists a $k$-isomorphism $\varphi:D \otimes_{k} K \simeq M_{d}(K)$ for $d:=\sqrt{\dim D}=\deg D$. Then the \textit{reduced norm} of $D$ is $\text{Nrd}_{D/k}:D\rightarrow k$,  $\text{Nrd}_{D/k}(a):= \det(\varphi(a \otimes 1))$. It is independent upon the choice of splitting field. Then $SL(1,D)=\{ a\in D: \, \text{Nrd}_{D/k}(a)=1\}$. \newline
Write $Aut_{k}(SL(1,D))$ for the group of $k$-automorphsim of $SL(1,D)$, the quotient $PSL(1,D)$ embeds in this group of automorphsims. \newline

\begin{Proposition} \label{ani~fin~ind}
     $[Aut_{k}(SL(1,D)):PSL(1,D)]<\infty$
\end{Proposition}

\begin{proof}
    One has that $D\hookrightarrow D\otimes \overline{\mathbb{Q}_{p}} \cong M_{d}(\overline{\mathbb{Q}_{p}})$ and $\quad SL(1,D)\otimes \overline{\mathbb{Q}_{p}}\cong SL_{d}(\overline{\mathbb{Q}_{p}})$ where $d$ is the degree of $D$. If $\varphi$ is a $k$-automorphism of $PSL(1,D)$ then $\varphi\otimes 1$ is a $k$-automorphism of $SL(d,\overline{\mathbb{Q}_{p}})$, hence can be written as $C_{g}\cdot s^{\epsilon} \cdot \sigma$, where $C_{g}$ is conjugation by some $g\in SL(d,\overline{\mathbb{Q}_{p}})$, $s$ is the non-trivial Dynkin automorphism of $SL(d)$ ($s(X)=(X^{t})^{-1}$), $\epsilon = 0,1$ and $\sigma\in Aut(\overline{\mathbb{Q}_{p}}/k)$. Note that $\sigma$ acts trivially on $k$, and hence acts trivially on $SL(1,D)$. \newline
    If $\epsilon = 0$, $C_{g}$ defines an automorphism of $SL(1,D)$. As it is a conjugation automorphism it also defines an automorphism (of algebras) of $D$. By the Skolem-Noether theorem \cite[Theorem 3.14]{FD93} it is inner, i.e. we can assume $g\in D^{\times}$. If $\epsilon =1$, $C_{g} \cdot T$ ($T$ is the transpose map) preserves $SL(1,D)$ and thus defines an isomorphism (of algebras) $D \to D^{op}$ (where $D^{op}$ is the opposite ring). Again by the Skolem-Noether theorem \cite[Theorem 3.14]{FD93}, we can assume that $g\in D^{\times}$. Thus, we can write $\varphi = C_{g}\cdot s^{\epsilon}$, where $g\in D^{\times}$, $s$ is the non-trivial Dynkin automorphism of $SL(d)$ and $\epsilon = 0,1$. \newline 
    Conjugating by $g\in D^{\times}$ is the same as conjugating by $g/\sqrt[d]{\text{Nrd}(g)}$ and thus, up to $d$-powers in $k^{\times}$ one can assume that $g\in SL(1,D)$. Hence
    \[
        [Aut_{k}(SL(1,D)):PSL(1,D)] \leq 2 \cdot |k^{\times}/(k^{\times})^{d}| < \infty
    \]
    As the subgroup $(k^{\times})^{d}$ of $d$ powers in $k^{\times}$ is of finite index in $k^{\times}$.\cite[Corollary II.5.8]{Ne13}.
\end{proof}

\begin{proof}[Proof of the first direction of the main theorem]
    Assume that $\varphi$ is an almost simple virtual endomorphism of $H$, write $H_{0}$ for the domain of $\varphi$. By restricting $\varphi$ if necessary, we can assume that $\varphi$ is injective and that $H_{0}$ is centerless. Moreover, as $H_{0}$ is finitely generated virtually pro-$p$ group, $\varphi$ is continuous \cite[Chapter I, \S 4.2, Exercise 6]{Se79}, and thus $\varphi$ defines an homeomorphism between $H_{0}$ and $\varphi(H_{0})$. In particular the dimension of $\varphi(H_{0})$ as a $p$-adic analytic group is equal to that of $H_{0}$ and $H$, and thus both $H_{0}$ and $\varphi(H_{0})$ are open compact subgroup of $SL(1,D)$. \newline
    By Pink's theorem \cite[Corollary 0.3]{Pi98} there exists a $k$-isomorphism $\Phi:SL(1,D)\to SL(1,D)$ and a field isomorphism $\sigma$ of $k$ so that
    $\varphi$ is the restriction of $\Phi\circ \sigma$ to $H_{0}$. By the previous Proposition (\ref{ani~fin~ind}), $H_{0}$ and $\varphi(H_{0})$ can be embedded as subgroups of finite index in $Aut(SL(1,D))$. Then there exists a normal (in $Aut(SL(1,D))$) finite index subgroup $N\subseteq H_{0}\cap \varphi(H_{0})$. Then $N = \varphi N \varphi^{-1}= \varphi(N)$, and thus $\varphi$ is not almost simple.
\end{proof}

\section{isotropic case}

For every simple algebraic group $\mathbf{G}$ over a local field $k$ there is a simplicial chamber complex $X_{G}$ called a the 'Bruhat-Tits building associated to $\mathbf{G}(k)$' (see for example \cite{Hi75},\cite{BT67}), this building is the non-Archimedean analog of the symmetric space associated with a real simple algebraic group. \newline
For our application we will need only the most basic properties of this building. In short, the simplicial structure of $X_{G}$ can be given as follows. Let $B\subseteq \mathbf{G}(k)$ be the normalizer of a Sylow pro-p subgroup of $G_{K}$ (note that all are conjugates), such $B$ is called an Iwahori subgroup. If a compact subgroup $P\subseteq \mathbf{G}(k)$ contains an Iwahori subgroup it is called parahoric. The vertices of the building $X_{G}$ are the maximal proper parahoric subgroups and a collection of vertices $\{P_{0},...,P_{s}\}$ defines an $s$-simplex if $\cap_{i=0}^{s}P_{i}$ is also a parahoric subgroup. The conjugation action of $\mathbf{G}(k)$ on $X_{G}$ is simplicial and the stabilizers of simplices are proper parahoric subgroups. The building $X_{G}$ (more precisely, its geometric realization) is a contractible space of dimension $\text{rank}_{k}\mathbf{G}$.\newline

Every automorphism $\varphi:\mathbf{G}(k)\to \mathbf{G}(k)$ defines an automorphism of the building $X_{G}$, $\varphi$ is called \textit{bounded} if there is a point $x\in X_{G}$ so that $\{\varphi^{n}x\}\subseteq X_{G}$ is bounded (note that in this case, for every bounded subset $Y\subseteq X_{G}$, $\{\varphi^{n}y:\, n\in \mathbb{Z},y\in Y\}$ is also bounded). It is called \textit{unbounded} if it is not bounded, in which case $\{\varphi^{n}(y)\}_{n=0}^{\infty}$ is unbounded for every point $y$ of $X_{G}$.\newline
If $\mathbf{G}(k)$ is not compact (i.e. $\text{rank}_{k}(G)\geq 1$), there are unbounded automorphisms of $X_{G}$. In fact, the conjugation action by every element $g\in \mathbf{G}(k)$, with $\overline{\langle g \rangle}$ not compact defines such an automorphism. \newline

Let $H\subseteq \mathbf{G}(k)$, be a compact open subgroup. Note that as in the previous section, by Pink's theorem \cite[Corollary 0.3]{Pi98} every virtual endomorphism of $H$ is the restriction of a genuine automorphsim of $\mathbf{G}(k)$, and on the other hand every automorphism of $\mathbf{G}(k)$ commensurate $H$, so there is a correspondence between virtual endomorphisms of $H$ and automorphisms of $\mathbf{G}(k)$. Hence, the following proposition will imply the second direction of the main theorem \ref{Mai~The}.

\begin{Proposition}
    Let $H\subseteq \mathbf{G}(k)$ be a compact open subgroup. Then every unbounded automorphism $\varphi:\mathbf{G}(k)\to \mathbf{G}(k)$ defines an almost simple virtual endomorphism of $H$.
\end{Proposition}

\begin{proof}
    Let $\varphi:\mathbf{G}(k) \to \mathbf{G}(k)$ be an automorphism, and by abuse of notation, say $\varphi=\varphi|_{H_{0}}:H_{0}\to H$ is a virtual endomorphism of $H$ defined by $\varphi$. \newline
    
    Assume that $\varphi$ is unbounded. Let $N\subseteq H_{0}$ be a non-central normal subgroup. It is well known that $N$ is of finite index in $H_{0}$, we give here a sketch of the argument. Since $H_{0}$ is open, its Lie algebra is equal to $\mathfrak{g}$ the $\mathbb{Q}_{p}$ Lie algebra of $G$, which is a simple Lie algebra. Now as $N$ is normal, its Lie algebra $\mathfrak{n}$ is an ideal of $\mathfrak{g}$. As the later is simple, either $\mathfrak{n}=\mathfrak{g}$ in which case $N$ is open in $G$ and hence of finite index in $H$, or $\mathfrak{n}=0$, in which case $N$ is finite. If $N$ is finite and normal, then its centralizer is open and so is Zariski dense, this implies that $N$ is indeed central. \newline
    As $H_{0}$ is a compact subgroup, $H_{0}\subseteq St(v_{0})$ for some vertex $v_{0}$ \cite[\S 14.7]{Ga97}. Assume that $N$ is $\varphi$-invariant (i.e. $\varphi(N)\subseteq N$). By the openness of $N$ it must be that $\varphi(N)=N$. Indeed, let $\mu$ be some Haar measure on $\mathbf{G}(k)$, the group $Aut(\mathbf{G}(k))$ acts on this measure by scalar multiplication, yielding a homomorphism $Aut(\mathbf{G}(k))\to \mathbb{R}_{>0}$, as the first group has finite abelianization and the second group is torsion free this map must be the trivial map, in particular $\varphi$ preserves the measure $\mu$. Thus, both $\varphi(N)\subseteq N$ are open with the same (non-zero) measure, and the inclusion must be an equality. This implies that if $v$ is fixed by $N$, then $N$ also fixes $\varphi(v)$, and by induction it fixes $\varphi^{n}(v)$ for every $n$. \newline
    The group $N\subseteq St(v_{0})$ is of finite index, and so its fixed point set in $X_{G}$ is bounded. This yields a contradiction, as $\{\varphi^{n}(v_{0})\}$ is an unbounded subset of the building. 
\end{proof}

In order to illuminate the virtual endomorphism produced in this fashion, we give a detailed geometric example for the case $H=SL(n,\mathbb{Z}_{p})$. Write $\mathbf{G}=SL(n)$, a concrete structure for the associated Bruhat-Tits building $X:=X_{G}$ of $SL(n,\mathbb{Q}_{p})$ can be given as follows (following \cite{Ga97} and \cite{PR93}). The vertices of $X$ are homothety equivalent classes of $\mathbb{Z}_{p}$-lattices in $\mathbb{Q}_{p}^{n}$. Define an incident relation $[L] = [M]$ if there are $L'\in [L]$, $M'\in [M]$ with $L'\subseteq M'$ and on the $\mathbb{Z}_{p}$-module $L'/M'$ one has $p\cdot L'/M'=0$ (so the quotient has the structure of a vector space over $\mathbb{F}_{p}$). It turns out that if $[L],[M]$ are incident, then any two representatives $L,M$ have the property that either $L\subseteq M$ or $L\supseteq M$. The maximal simplices (called 'chambers') of the simplicial complex defined by this incident geometry are in bijection with ascending chains of lattices
\[
... \subseteq L_{-1} \subseteq L_{0} \subseteq ... \subseteq L_{n-1} \subseteq L_{n} \subseteq ...
\]
with periodicity $L_{i+n}=pL_{i}$ for all indices $i$ and where the quotients $L_{i+1}/L_{i}$ are all one-dimensional $\mathbb{F}_{p}$ vector spaces. \newline
Let $\{e_{1},...,e_{n}\}$ be the standard basis for $\mathbb{Q}_{p}^{n}$, and consider the $\mathbb{Z}_{p}$-lattices $\Lambda_{i}$ with bases $e_{1},...,e_{n-i},pe_{n-i+1},...,pe_{n}$ for $0 \leq i \leq n-1$, their homothety classes correspond to a maximal simplex (chamber) $C$ in the building with stabilizer the Iwahori subgroup
\[
B = \{ x=(x_{ij})\in SL(n,\mathbb{Z}_{p}): \, x_{ij} \equiv 0 \mod p \text{ for } i>j\}
\]
which is the normalizer of the Sylow pro-p subgroup of matrices $x=(x_{ij})\in SL(n,\mathbb{Z}_{p})$ for which $x_{ii}\equiv 1 \mod p$ and $x_{ij}\equiv 0\mod p$ for all $1\leq i<j\leq n$. Further, consider the subcomplex $A$ consisting of all simplices $\sigma$ with vertices $[L]$ which are homothety classes of lattices with a representative $L$ expressible as $L=L_{1}+...+L_{n}$ where $L_{i}$ is a $\mathbb{Z}_{p}$-lattice in the line $\mathbb{Q}_{p}e_{i}$ (this subcomplex is called an 'apartment' of the building). As a geometric space, this subcomplex $A$ is isomorphic $\mathbb{R}^{n-1}$ with simplicial structure resulting by the cut outs of a certain set of hyperplanes. The diagonal torus $S$ of $SL(n,\mathbb{Q}_{p})$ action preserves $A$ and acts on $A \cong \mathbb{R}^{n-1}$ by translations which fix the set of these hyperplanes. Let $s=\text{diag}(s_{1},...,s_{n})\in S$ be such that $\text{val}_{p}(s_{i+1}/s_{i})$ are all distinct and non-zero. Then the (infinite) line going through $[\Lambda_{0}]=[\mathbb{Z}_{p}(e_{1}\oplus \cdots \oplus e_{n})]$ and $s.[\Lambda_{0}]$ does not lie in any of the hyperplanes, this implies that we can find a point $x$ close to $[\Lambda_{0}]$ in the interior of the chamber $C$ so the line going through $x$ and $s.x$ will only intersect chambers of $A$ and the co-dimension 1 faces and so the convex hull of all chambers that see this line (convex in the sense of chamber complexes) is the whole apartment $A\cong \mathbb{R}^{n-1}$, see Figure \ref{fig~2}. (See also \cite{Mo95}).

Let $\varphi$ be conjugation by $s$, and $H_{0}$ be the principle congruence subgroup of level $\max\{\text{val}_{p}(s_{j}/s_{i}):\, 1\leq i,j \leq n \}$, so $\varphi(H_{0})\subseteq SL(n,\mathbb{Z}_{p})$, and $\varphi:H_{0}\to SL(n,\mathbb{Z}_{p})$ is a virtual endomorphism. Note that $H_{0}\subseteq B$ and if $N\subseteq H_{0}$ is a non-central normal subgroup, then it fixes the chamber $C$ pointwise. By the above analysis of the action of $s$ on the apartment $A$ we see that $N$ must fix $A$ pointwise. Thus $N$ must be contained in $T$, the diagonal torus of $SL(n,\mathbb{Z}_{p})$. Such a subgroup cannot be normal in $H_{0}$, either by noticing it is not of finite index, or more concretely, it is not preserved under conjugation of   (small) unipotent elements in $H_{0}$.

\pgfmathsetmacro{\cols}{12}
\pgfmathsetmacro{\rows}{8}
\pgfmathsetmacro{\slant}{cot(60)}
\pgfmathsetmacro{\height}{0.5 * \rows * tan(60)}

\begin{figure}[H]
\begin{tikzpicture}
    \clip       (0, 0) rectangle (\cols, \height-0.5);

    \pgfmathsetmacro{\from}{-2 *\cols}
    \pgfmathsetmacro{\to}{2 * \cols}
    \foreach\i in {\from, ..., \to} {
        \draw[xslant=\slant]  (\i, 0) -- (\i, \height);
        \draw[xslant=-\slant] (\i, 0) -- (\i, \height);
    }

    \foreach\j in {0, ..., \rows} {
        \pgfmathsetmacro{\y}{0.5 * \j * tan(60)}
        \draw (0, \y) -- (\cols, \y);
    }
    \draw[blue,shorten >=-20cm,shorten <=-10cm] (1.333337,1.333337)  -- (3.333337,2.333337) ;

\end{tikzpicture}
\caption{The apartment $A$ for $\mathbf{G}=SL(3)$ and the line through $x$ and $s.x$}\label{fig~2}
\end{figure}
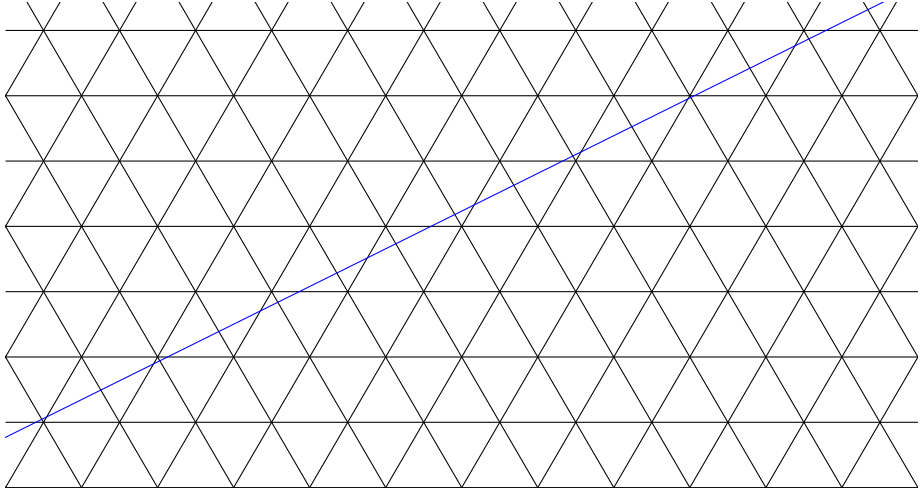

\bibliographystyle{acm}
\bibliography{References}

\nocite{*}

\end{document}